\documentclass{amsart}
\usepackage[latin1]{inputenc}
\usepackage{amsmath}
\usepackage{amsthm}
\usepackage{amsfonts}
\usepackage{amssymb}
\usepackage[T1]{fontenc}
\usepackage{float}
\usepackage[all]{xy}
\usepackage{tikz}
\usetikzlibrary{arrows}
\usepackage[hmargin=3cm]{geometry}
\usepackage{hyperref}
\usepackage{bbm}

\theoremstyle{plain}
\newtheorem{theorem}{Theorem}
\newtheorem{prop}[theorem]{Proposition}
\newtheorem{lem}[theorem]{Lemma}
\newtheorem{corol}[theorem]{Corollary}

\theoremstyle{definition}

\def\dim{{\rm{dim}}}
\def\ddim{{\mathbf{dim}\,}}
\def\den{{\mathbf{den}\,}}

\def\c{\mathbf{c}}
\def\k{\mathbf{k}}

\def\kQ{\k Q}

\def\CC{{\mathcal{C}}}
\def\T{{\mathcal {T}}}

\def\Z{{\mathbb{Z}}}

\def\add{{\rm{add}}\,}
\def\ind{{\rm{ind}}}

\def\Ext{{\rm{Ext}}}
\def\End{{\rm{End}}}
\def\Hom{{\rm{Hom}}}
\def\Tor{{\rm{Tor}}}

\def\modd{{\mathrm{mod-}}}

\def\ens#1{\left\{ #1 \right\}}
\def\fl{{\longrightarrow}\,}

\title{Modules over cluster-tilted algebras determined by their dimension vectors}
\author{Ibrahim Assem and Grégoire Dupont}
\address{Université de Sherbrooke, 2500 Boul. de l'université, J1K 2R1 Sherbrooke QC, Canada.}
\email{Ibrahim.Assem@USherbrooke.ca}
\address{Université Denis Diderot Paris 7 -- IMJ--PRG, 175 rue du chevaleret, 75013 Paris, France.}
\email{dupontg@math.jussieu.fr}
\date{\today}
\begin{document}

\begin{abstract}
	We prove that indecomposable transjective modules over cluster-tilted algebras are uniquely determined by their dimension vectors. Similarly, we prove that for cluster-concealed algebras, rigid modules lifting to rigid objects in the corresponding cluster category are uniquely determined by their dimension vectors. Finally, we apply our results to a conjecture of Fomin and Zelevinsky on denominators of cluster variables.
\end{abstract}

\maketitle


\section*{Introduction}
	One of the most important results of modern-day representation theory of algebras is the now classical theorem of Gabriel stating that a (finite, connected, acyclic) quiver $Q$ is representation-finite, that is, admits only finitely many isomorphism classes of indecomposable representations, if and only if its underlying diagram is a Dynkin diagram, see \cite{Gabriel:theorem}. The sufficiency part is proven by showing that the map sending each indecomposable representation to its dimension vector (which records the number of its composition factors) establishes a bijection between the set of isomorphism classes of indecomposable representations of $Q$ and the set of positive roots of the Dynkin diagram underlying $Q$. Since then, it has become a standard question in representation theory to identify which indecomposable modules are uniquely determined up to isomorphism by their composition series, or equivalently by their dimension vectors.

	In this paper, we study this question for the class of cluster-tilted algebras which was introduced by Buan, Marsh and Reiten in \cite{BMR1} as a by-product of the theory of cluster algebras of Fomin and Zelevinsky \cite{cluster1}. More precisely, we are interested in the question whether the indecomposable rigid modules (that is, those without self-extensions) are uniquely determined by their dimension vectors.

	In order to state our main result, we first recall that the Auslander-Reiten quiver of a cluster-tilted algebra always has a unique component containing local slices \cite{ABS:slices}, which coincides with the whole Auslander-Reiten quiver whenever the cluster-tilted algebra is representation-finite. This component is called the \emph{transjective} component and an indecomposable module lying in it is called a \emph{transjective} module. With this terminology, our main result may be stated as follows~:

	\begin{theorem}\label{theorem:transjective}
		Let $B$ be a cluster-tilted algebra and $M,N$ be indecomposable transjective $B$-modules. Then $M$ is isomorphic to $N$ if and only if $M$ and $N$ have the same dimension vector.
	\end{theorem}

	As we shall see in Section \ref{section:cex}, the statement of the theorem does not hold true without the assumption that $M$ and $N$ are transjective. 

	One important special case of our theorem is already known~: Ringel \cite{Ringel:clusterconcealed} and, independently Geng and Peng \cite{GengPeng} have proved that, over a representation-finite cluster-tilted algebra, indecomposable modules are uniquely determined by their dimension vectors. Our result however makes no assumption on the representation type of the algebra.

	For our second result, we consider one important subclass of cluster-tilted algebras, that of the cluster-concealed algebras studied in \cite{Ringel:clusterconcealed}. This class consists of those cluster-tilted algebras which are endomorphism algebras of transjective tilting objects in the cluster category. We are now able to state our result~:
	\begin{prop}\label{prop:concealed}
		Let $B$ be a cluster-concealed algebra and $M,N$ be indecomposable rigid $B$-modules lifting to rigid objects in the corresponding cluster category. Then $M$ is isomorphic to $N$ if and only if $M$ and $N$ have the same dimension vector.
	\end{prop}

	In the case where $B$ is wild, we do not know whether all indecomposable rigid $B$-modules lift to rigid objects in the corresponding cluster category. This, however, is the case if $B$ is tame.
	\begin{corol}\label{corol:affineconcealed}
		Let $B$ be a cluster-concealed algebra which is either representation-finite or tame and let $M,N$ be two indecomposable rigid $B$-modules. Then $M$ is isomorphic to $N$ if and only if $M$ and $N$ have the same dimension vector.
	\end{corol}

	We next apply our results in order to obtain partial results towards a classical conjecture stating that different cluster variables, when expressed as Laurent polynomials in any cluster, have different denominator vectors, see \cite[Conjecture 4.17]{FZ:cdm03}. 

	We recall that an \emph{acyclic cluster algebra} $\mathcal A$ is a skew-symmetric cluster algebra with coefficients in an arbitrary semifield associated with a quiver which is equivalent under mutation to an acyclic quiver $Q$. It follows from the Laurent phenomenon \cite{cluster1} that for any cluster $\c=(c_i, i \in Q_0)$ of $\mathcal A$, every cluster variable $x \in \mathcal A$ can be written as 
	$$x = \frac{P(c_i, i \in Q_0)}{\prod_{i \in Q_0} c_i^{d_i}}$$
	where $P(c_i, i \in Q_0)$ is a polynomial not divisible by any $c_i$ and where $d_i \in \Z$ for any $i \in Q_0$. The $Q_0$-tuple 
	$$\den(x) = (d_i)_{i \in Q_0}$$ is called the \emph{denominator vector} of $x$ in the cluster $\c$. 

	If $\CC$ is the cluster category of $Q$, then it known that there is a bijection, the \emph{cluster character}, between the set of cluster variables in $\mathcal A$ and the set of isomorphism classes of indecomposable rigid objects in $\CC$ inducing a bijection between the set of clusters in $\mathcal A$ and the set of isomorphism classes of tilting objects in $\CC$, see \cite{BMRT,CK2,FK}. A cluster variable (or a cluster, respectively) in $\mathcal A$ is called \emph{transjective} if it corresponds to an indecomposable transjective object (or a transjective tilting object, respectively) in $\CC$. If $\mathcal A$ is of finite type (in the sense of \cite{cluster2}) or of rank two, then every cluster variable, and thus every cluster, is transjective. 

	In this terminology, our Theorem \ref{theorem:transjective} and our Proposition \ref{prop:concealed} imply the following two corollaries~:
	\begin{corol}\label{corol:dentransj}
		Let $\mathcal A$ be an acyclic cluster algebra and $\CC$ the corresponding cluster category. Then the transjective cluster variables in $\mathcal A$ are uniquely determined by their denominator vectors in any cluster of $\mathcal A$.
	\end{corol}

	\begin{corol}\label{corol:denaffine}
		Let $\mathcal A$ be an acyclic cluster algebra and $\CC$ the corresponding cluster category. Then the cluster variables in $\mathcal A$ are uniquely determined by their denominator vectors in any transjective cluster of $\mathcal A$.
	\end{corol}

	The article is organised as follows. In Section \ref{section:background} we recall the necessary background and notations. In Section \ref{section:transjective} we prove Theorem \ref{theorem:transjective} and Corollary \ref{corol:dentransj} and we show that the transjectivity assumption cannot be removed from Theorem \ref{theorem:transjective}. Finally, Section \ref{section:clusterconcealed} presents the proofs of Proposition \ref{prop:concealed} and Corollaries \ref{corol:affineconcealed} and \ref{corol:denaffine}.

\section{Background and notations}\label{section:background}
	Throughout the article, $\k$ denotes an algebraically closed field and every algebra is a $\k$-algebra. 

	Given a finite dimensional algebra $A$, we denote by $\modd A$ the category of finitely generated right $A$-modules and by $\ind-A$ a full subcategory of $\modd A$ consisting of a complete set of representatives of isomorphism classes of indecomposable $A$-modules. For any $A$-module $M$, we denote by $\ddim M$ its dimension vector which we view as an element in $\Z^n$ where $n$ is the number of simple $A$-modules. An $A$-module $M$ is called \emph{rigid} if $\Ext^1_A(M,M)=0$. A \emph{tilting module} $T$ in $\modd A$ is a basic rigid $A$-module of projective dimension at most one and with $n$ non-isomorphic direct summands $T_1, \ldots, T_n$. If $T$ is such a tilting module, then it follows from \cite{HU:almostcomplete} that for any $k$ such that $1 \leq k \leq n$ the \emph{almost complete} module $\overline T = \bigoplus_{i \neq k} T_i$ admits at most two \emph{complements}, that is, non-isomorphic indecomposable $A$-modules $T_k$ and $T_k^*$ such that $\overline T \oplus T_k$ and $\overline T \oplus T_k^*$ are tilting modules. Moreover, if $\overline T$ is sincere, then there exist exactly two such complements. 

	Given a quiver $Q$, we denote by $Q_0$ its set of points and by $Q_1$ its set of arrows. We always assume that $Q_0$ and $Q_1$ are finite sets and we let $n=|Q_0|$. If $Q$ is acyclic, we denote by $\kQ$ its path algebra over $\k$. It is a finite dimensional hereditary algebra whose bounded derived category $D^b(\modd \kQ)$ is triangulated with suspension functor $[1]$ and Auslander-Reiten translation $\tau$. The \emph{cluster category} $\CC$ of $Q$, introduced in \cite{BMRRT}, is the orbit category in $D^b(\modd \kQ)$ of the autoequivalence $\tau^{-1}[1]$. It is a triangulated category \cite{K} whose suspension functor $\Sigma$ is induced by the shift in $D^b(\modd \kQ)$ or, equivalently, by the Auslander-Reiten translation. Moreover, the category $\CC$ is \emph{Calabi-Yau of dimension 2}, that is, there is a bifunctorial isomorphism
	$$\Hom_{\CC}(M,\Sigma^2 N) \simeq D\Hom_{\CC}(N,M)$$
	for any objects $M$ and $N$ in $\CC$ where $D= \Hom_{\k}(-,\k)$ is the standard duality.

	The Auslander-Reiten quiver $\Gamma(\CC)$ of $\CC$ contains a unique connected component containing the indecomposable projective $\kQ$-modules. This component is called the \emph{transjective component} of $\Gamma(\CC)$. If $Q$ is of Dynkin type, this is the only connected component. Otherwise, there are infinitely many other connected components, which are called the \emph{regular components} of $\Gamma(\CC)$. An object in $\CC$ is called \emph{transjective} if all its indecomposable direct summands belong to the transjective component. Every indecomposable transjective object $M$ is \emph{rigid} in $\CC$, that is, $\Ext^1_{\CC}(M, M) = 0$. 

	A \emph{tilting object} $T$ in $\CC$ is a basic object such that for any object $X$ in $\CC$, we have $\Ext^1_{\CC}(T,X)=0$ if and only if $X \in \add(T)$ (these are also referred to as \emph{cluster-tilting objects} in the literature). Let $T$ be such a tilting object. Then it is well known that $T$ has $n$ indecomposable direct summands which we denote by $T_i$, with $1 \leq i \leq n$. 

	We denote by $B = \End_{\CC}(T)$ the \emph{cluster-tilted algebra} corresponding to $T$. We denote by $(\add(\Sigma T))$ the ideal of morphisms in $\CC$ factoring through the subcategory $\add (\Sigma T)$. Then it is proved in \cite{BMR1} that the functor $\Hom_{\CC}(T,-)$ induces an equivalence of categories~
	$$\Hom_{\CC}(T,-): \CC/(\add(\Sigma T)) \xrightarrow{\sim} \modd B.$$
	If $M$ is a $B$-module, a \emph{lift} of $M$ in $\CC$ is an object $\widetilde M$ in $\CC$ such that $\Hom_{\CC}(T,\widetilde M) \simeq M$ in $\modd B$. The \emph{transjective} $B$-modules are those having a transjective lift in $\CC$.

	Fix $1 \leq k \leq n$. We let $\overline T = \bigoplus_{i \neq k} T_i$ be the corresponding \emph{almost complete object}. Then it is proved in \cite{BMRRT} that there exists a unique object $T_k^*$ in $\CC$ such that $T_k^* \not \simeq T_k$ and such that $T'=\overline T \oplus T_k^*$ is a tilting object in $\CC$, called \emph{the mutation of $T$ at $T_k$}. The pair $(T_k,T_k^*)$ is called an \emph{exchange pair} in $\CC$ and it is known that $\Ext^1_{\CC}(T_k,T_k^*) \simeq \k$ so that, up to isomorphism, there exist exactly two non-split triangles
	$$T_k \fl E \fl T_k^* \fl \Sigma T_k \text{ and } T_k^* \fl E' \fl T_k \fl \Sigma T_k^*,$$
	called the \emph{exchange triangles} associated with the exchange pair $(T_k,T_k^*)$.

	Following \cite{BMR3}, an indecomposable rigid object $M$ in $\CC$ is called \emph{compatible with the exchange pair} $(T_k,T_k^*)$ if either $M \simeq \Sigma^{-1}T_k$, or $M \simeq \Sigma^{-1}T_k^*$ or
	$$\dim_{\k} \Hom_{\CC}(M,T_k) + \dim_{\k} \Hom_{\CC}(M,T_k^*) = \max \ens{\dim_{\k} \Hom_{\CC}(M,E), \dim_{\k} \Hom_{\CC}(M,E')}.$$

\section{Transjective modules and cluster variables}\label{section:transjective}
	\begin{lem}\label{lem:dim}
		Let $(T_k,T_k^*)$ be an exchange pair in $\CC$ as above and let $M$ be an indecomposable rigid object in $\CC$. Then $M$ is compatible with the exchange pair $(\Sigma^2 T_k, \Sigma^2 T_k^*)$ if and only if either $M \simeq \Sigma T_k$ or $M \not \simeq \Sigma T_k^*$ or 
			$$\dim_{\k} \Hom_{\CC}(T_k,M) + \dim_{\k} \Hom_{\CC}(T_k^*,M) = \max \ens{\dim_{\k} \Hom_{\CC}(E,M), \dim_{\k} \Hom_{\CC}(E',M)}.$$
	\end{lem}
	\begin{proof}
		It follows from the definition that $M$ is compatible with the exchange pair $(\Sigma^2 T_k, \Sigma^2 T_k^*)$ if and only if either $M \simeq \Sigma T_k$ or $M \simeq \Sigma T_k^*$ or 
		$$
			\dim_{\k} \Hom_{\CC}(M,\Sigma^2 T_k)  + \dim_{\k} \Hom_{\CC}(M,\Sigma^2 T_k^*) 
							= \max \ens{\dim_{\k} \Hom_{\CC}(M,\Sigma^2 E), \dim_{\k} \Hom_{\CC}(M,\Sigma^2 E')}
		$$
		because $$\Sigma^2 T_k \fl \Sigma^2 E \fl \Sigma^2 T_k^* \fl \Sigma^3 T_k \text{ and }\Sigma^2 T_k^* \fl \Sigma^2 E \fl \Sigma^2 T_k \fl \Sigma^3 T_k^*$$
		are exchange triangles.

		Now, since $\CC$ is Calabi-Yau of dimension 2, we have $\dim_{\k} \Hom_{\CC}(X,\Sigma^2 Y) = \dim_{\k} \Hom_{\CC}(Y,X)$ for any objects $X,Y$ in $\CC$.

		Therefore, if $M \not \simeq \Sigma T_k$ and $M \not \simeq \Sigma T_k^*$, then $M$ is compatible with the exchange pair $(\Sigma^2 T_k, \Sigma^2 T_k^*)$ if and only if
		$$\dim_{\k} \Hom_{\CC}(T_k,M) + \dim_{\k} \Hom_{\CC}(T_k^*,M) = \max \ens{\dim_{\k} \Hom_{\CC}(E,M), \dim_{\k} \Hom_{\CC}(E',M)}.$$
	\end{proof}

	\begin{lem}\label{lem:induction}
		Let $T=\bigoplus_{i=1}^n T_i$ be a tilting object in $\CC$ and let $T' = \overline T \oplus T_k^*$ be its mutation at $k$, where $\overline T = \bigoplus_{i \neq k} T_i$. Let $M$ and $N$ be two indecomposable rigid objects in $\CC$ which are compatible with the exchange pair $(\Sigma^2 T_k,\Sigma^2 T^*_k)$. Assume that $M,N \not \in \add(\Sigma T)$ and that $\ddim \Hom_{\CC}(T,M) = \ddim \Hom_{\CC}(T,N)$. Then either $M \simeq N \simeq \Sigma T_k^*$ or $M,N \not \in \add(\Sigma T')$ and $\ddim \Hom_{\CC}(T',M) = \ddim \Hom_{\CC}(T',N)$.
	\end{lem}
	\begin{proof}
		Assume first that $M$ belongs to $\add(\Sigma T')$. Since $T'=\overline T \oplus T_k^*$ and $M \not \in \add(\Sigma T)$, we get $M \simeq \Sigma T_k^*$. Moreover, since $T'$ is a tilting object, it is rigid and thus
		$$\dim_{\k} \Hom_{\CC}(T_i,M) = \dim_{\k} \Hom_{\CC}(T_i,\Sigma T_k^*) = \delta_{ik}.$$
		Hence, we also have $\dim_{\k} \Hom_{\CC}(T_i,N) = \delta_{ik}$. Since $N$ is indecomposable and rigid, it follows that $\overline T \oplus \Sigma^{-1}N$ is a tilting object. Since the almost complete object $\overline T $ admits only two complements which are $T_k$ and $T_k^*$, it follows that $\Sigma^{-1}N \simeq T_k$ or $\Sigma^{-1}N \simeq T_k^*$. By assumption, $N \not \in \add(\Sigma T)$ and therefore, $N \simeq \Sigma T_k^* = M$. 

		By symmetry, we can assume that $M,N \not \in \add(\Sigma T')$. Then, since $M$ and $N$ are compatible with the exchange pair $(\Sigma^2 T_k,\Sigma^2 T^*_k)$, it follows from Lemma \ref{lem:dim} that
		$$\dim_{\k} \Hom_{\CC}(T_k,M) + \dim_{\k} \Hom_{\CC}(T_k^*,M) = \max \ens{\dim_{\k} \Hom_{\CC}(E,M), \dim_{\k} \Hom_{\CC}(E',M)}$$
		and
		$$\dim_{\k} \Hom_{\CC}(T_k,N) + \dim_{\k} \Hom_{\CC}(T_k^*,N) = \max \ens{\dim_{\k} \Hom_{\CC}(E,N), \dim_{\k} \Hom_{\CC}(E',N)}.$$
		By assumption, $\ddim \Hom_{\CC}(T,M) = \ddim \Hom_{\CC}(T,N)$ so that $\dim_{\k} \Hom_{\CC}(T_i,N) = \dim_{\k} \Hom_{\CC}(T_i,N)$ for any $i$ such that $1 \leq i \leq n$. And since $E,E'$ are in $\add(\overline T)$, we have $\dim_{\k} \Hom_{\CC}(E,N) = \dim_{\k} \Hom_{\CC}(E,M)$ and $\dim_{\k} \Hom_{\CC}(E',M) = \dim_{\k} \Hom_{\CC}(E',N)$. Therefore, $\dim_{\k} \Hom_{\CC}(T_k^*,M) = \dim_{\k} \Hom_{\CC}(T_k^*,N)$ and thus $\ddim \Hom_{\CC}(T',M) = \ddim \Hom_{\CC}(T',N)$.
	\end{proof}

	\subsection{Proof of Theorem \ref{theorem:transjective}}
		Let $\CC$ be a cluster category and $T$ a tilting object in $\CC$. Let $\widetilde M$ and $\widetilde N$ be two indecomposable transjective objects in $\CC$. They are thus rigid and by \cite[Corollary 4.2]{BMR3}, they are compatible with any exchange pair in $\CC$. Since $\Hom_{\CC}(T,-)$ induces an equivalence $\CC/(\add(\Sigma T)) \xrightarrow{\sim} \modd B$, in order to prove the theorem, it is enough to prove that if $\widetilde M, \widetilde N \not \in \add(\Sigma T)$ are such that $\ddim \Hom_{\CC}(T,\widetilde M) = \ddim \Hom_{\CC}(T,\widetilde N)$, then $\widetilde M \simeq \widetilde N$ in $\CC$. It is known that two tilting object in $\CC$ can be joined by a sequence of mutations, see \cite{BMRRT}. In particular, there exists a sequence of mutations $T \fl T' \fl \cdots \fl T^{(l)} =\kQ$ joining $T$ to the tilting object $\kQ$ in $\CC$.

		Let $\overline T = \bigoplus_{i \neq k} T_i$ and let $T' = \overline T \oplus T_k^*$ be the mutation of $T$ at $k$. By Lemma \ref{lem:induction}, either $\widetilde M \simeq \widetilde N \simeq \Sigma T_k^*$ and the claim is proved, or $\widetilde M,\widetilde N \not \in \add(\Sigma T')$ and $\ddim \Hom_{\CC}(T',\widetilde M) = \ddim \Hom_{\CC}(T',\widetilde N)$. By induction on $l$, either we have obtained $\widetilde M \simeq \widetilde N$ at some point, or we end up with $\ddim \Hom_{\CC}(\kQ,\widetilde M) = \ddim \Hom_{\CC}(\kQ,\widetilde N)$. Therefore, $\Hom_{\CC}(\kQ,\widetilde M)$ and $\Hom_{\CC}(\kQ,\widetilde N)$ are two indecomposable rigid modules with the same dimension vectors over the hereditary algebra $\kQ$. Since for hereditary algebras, it is well-known that rigid modules are uniquely determined by their dimension vectors, it follows that $\Hom_{\CC}(\kQ,\widetilde M)$ and $\Hom_{\CC}(\kQ,\widetilde N)$ are isomorphic modules and therefore, so are their respective indecomposable lifts $\widetilde M$ and $\widetilde N$ in $\CC$. \hfill \qed

	\subsection{A counter-example without the transjectivity assumptions}\label{section:cex}
		We now show an example of two non-isomorphic indecomposable non-transjective rigid modules over a cluster-tilted algebra which have the same dimension vectors and which do not lift to rigid objects in the cluster category.

		Consider the cluster-tilted algebra $B$ of affine type $\widetilde A_{2,1}$ given by the quiver 
		$$
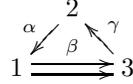
\xymatrix@=1em{
				&& 2 \ar[ld]_{\alpha}\\
			& 1 \ar@<+2pt>[rr]^{\beta} \ar@<-2pt>[rr] && 3 \ar[lu]_{\gamma}
		}$$
		with relations $\alpha\beta = \beta\gamma = \gamma\alpha = 0$. 

		The Auslander-Reiten quiver of the corresponding cluster category contains exactly one exceptional tube which is of rank two. In $\modd B$, the image of this tube is 
		
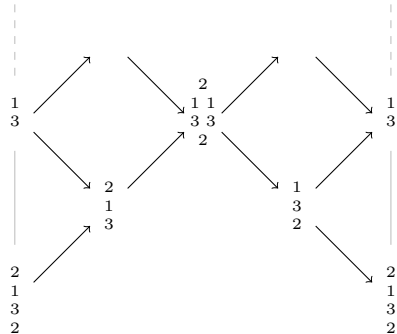
\begin{center}
			\begin{tikzpicture}[scale = .5]
				\tikzstyle{every node}=[font=\tiny]
				\fill (0,0) node {$\begin{array}{c} 2\\1\\3\\2 \end{array}$};

				\draw[->] (.5,.5) -- (2,2);

				\fill (2.5,2.5) node {$\begin{array}{c} 2\\1\\3 \end{array}$};

				\draw[->] (3,3) -- (4.5,4.5);

				\fill (0,5) node {$\begin{array}{c} 1\\3 \end{array}$};

				\draw[->] (.5,4.5) -- (2,3);

				\fill (7.5,2.5) node {$\begin{array}{c} 1\\3\\2 \end{array}$};

				\draw[->] (8,3) -- (9.5,4.5);

				\fill (5,5) node {$\begin{array}{c} 2\\1 ~ 1 \\ 3 ~ 3 \\ 2\end{array}$};

				\draw[->] (5.5,4.5) -- (7,3);

				\fill (10,5) node {$\begin{array}{c} 1\\3 \end{array}$};

				\draw[->] (8,2) -- (9.5,.5);

				\fill (10,0) node {$\begin{array}{c} 2\\1\\3\\2 \end{array}$};

				\draw[->] (.5,5) -- (2,6.5);

				\draw[->] (5.5,5) -- (7,6.5);

				\draw[->] (3,6.5) -- (4.5,5);

				\draw[->] (8,6.5) -- (9.5,5);

				\draw[gray!50] (0,1.5) -- (0,4);
				\draw[gray!50,dashed] (0,6) -- (0,8);

				\draw[gray!50] (10,1.5) -- (10,4);
				\draw[gray!50,dashed] (10,6) -- (10,8);
			\end{tikzpicture}
		\end{center}
		
		Therefore, if we set 
		$$M = \begin{array}{c} 1\\3\\2 \end{array} \text{ and } N = \begin{array}{c} 2\\1\\3 \end{array},$$
		then $M$ and $N$ have the same dimension vector $(111)$ and $\Ext^1_{B}(N,N) \simeq D\overline{\Hom}_{B}(N,M) = 0$, because every nonzero morphism from $M$ to $N$ factors through the injective module at 2. Similarly, $\Ext^1_{B}(M,M) = 0$. However, $M$ and $N$ are not isomorphic.

		Finally, we recall that in \cite{Smith:lifts}, the author proved that any rigid module of projective dimension at most one over a cluster-tilted algebra lifts to a rigid object in the cluster category, see also \cite{FL:lifts}. As it appears in this example, this does not hold true without the assumption that the projective dimension is at most one. Indeed, the lifts of $M$ and $N$ in the cluster category are objects of quasi-length two in an exceptional tube of rank two. In particular, they are not rigid.

	\subsection{Proof of Corollary \ref{corol:dentransj}}
		Let $\mathcal A$ be an acyclic cluster algebra. It is known, see for instance \cite[\S 7]{cluster4}, that the denominator vectors of the cluster variables in $\mathcal A$ do not depend on the choice of the coefficient system. Therefore, without loss of generality, we can assume that $\mathcal A$ is coefficient-free. 

		Let $Q$ be an acyclic quiver which is mutation-equivalent to the quiver in the initial seed of $\mathcal A$ and $\CC$ the cluster category of $Q$. For any tilting object $T$ and any indecomposable rigid object $M$ in $\CC$, we denote by $X^T_M$ the cluster variable corresponding to $M$ expressed as a Laurent polynomial in the cluster corresponding to $T$ under the bijection mentioned in the introduction. We denote by $\den(X^T_M)$ the denominator vector of the Laurent polynomial $X^T_M$, which we view as a vector in $\Z^n$.

		Let $x$ and $y$ be two transjective cluster variables in $\mathcal A$ and assume that they have the same denominator vectors in a given cluster $\mathbf c$. Let $T$ be the tilting object in $\CC$ corresponding to $\mathbf c$ and $M$ and $N$ transjective objects in $\CC$ such that $x = X^T_M$ and $y = X^T_N$. By assumption, we have $\den(X^T_M) = \den(X^T_N)$. 

		If $M$ or $N$ are in $\add(\Sigma T)$, then $X^T_M$ and $X^T_N$ are in the cluster $\mathbf c$ and therefore, $X^T_M = X^T_N$. If they are not in $\add(\Sigma T)$, then since $M$ and $N$ are transjective, it follows from \cite{BMR3} that $\den(X^T_M)= \ddim \Hom_{\CC}(T,M)$ and $\den(X^T_N)= \ddim \Hom_{\CC}(T,N)$. It follows that $\Hom_{\CC}(T,M)$ and $\Hom_{\CC}(T,N)$ are indecomposable rigid transjective $B$-modules with the same dimension vector and thus, Theorem \ref{theorem:transjective} implies that $\Hom_{\CC}(T,M) = \Hom_{\CC}(T,N)$. Therefore, $M \simeq N$ in $\CC$ and thus $x = X^T_M = X^T_N =y$. \hfill \qed

\section{Cluster-concealed algebras}\label{section:clusterconcealed}
	In this section we focus on the case of transjective tilting objects. We first prove that the full subgraph of the mutation graph of a cluster category consisting of the transjective tilting objects is connected.

	Let $A$ be a hereditary algebra and $T$ be a preinjective tilting $A$-module. Letting $\T(T) = \ens{M \in \modd A \ | \ \Ext^1_A(T,M)=0}$ be the torsion class associated with the tilting module $T$, the preinjectivity of $T$ implies that the subcategory $\ind-\T(T) = \ind-A \cap \T(T)$ contains only finitely many objects.
	\begin{prop}\label{prop:muttransjective}
		Let $A$ be a hereditary algebra and $T$ a preinjective tilting module. Then there exists a sequence of tilting $A$-modules
		$$\xymatrix{T = T^{(0)} \ar[r] &  T^{(1)} \ar[r] & \cdots \ar[r] & T^{(t)} = DA_A}$$
		where each $T^{(i)}$ is a tilting $A$-module differing from $T^{(i+1)}$ by a single indecomposable summand and such that $\T(T^{(i)}) \supsetneqq \T(T^{(i+1)})$ for every $i<t$.
	\end{prop}
	\begin{proof}
		If $T = DA_A$, there is nothing to prove. Otherwise, $T$ has a non-injective indecomposable summand. This implies that $T$ has a non-injective indecomposable summand $T_0$ such that there is no path $T_i \rightsquigarrow T_0$ with $T_i \not \simeq T_0$ another indecomposable summand of $T$. We claim that $\overline T = \bigoplus_{i \neq 0}T_i$ is sincere.

		Notice that the hypothesis means that $\Hom_A(T,T_0)$ is a simple projective module over the tilted algebra $C = \End_A(T)$ and moreover it is non-injective (because $C$ is connected). Hence the left minimal almost split morphism starting with $\Hom_A(T,T_0)$ has a projective target and we have an almost split sequence in $\modd C$
		$$0 \fl \Hom_A(T,T_0) \fl \Hom_A(T,T^*) \fl Z \fl 0$$
		with $T^* \in \add \overline T$. We claim that the canonical morphism $T_0 \fl T^*$ is injective. In order to prove this, we first observe that $Z_C$ belongs to the class $\mathcal Y(T) = \ens{Y_C \ | \ \Tor^C_1(Y,T)=0}$. Indeed, if this is not the case, then, because the torsion pair $(\mathcal X(T),\mathcal Y(T))$ in $\modd C$ is split, and $Z$ is indecomposable then it would belong to the class $\mathcal X(T) = \ens{X_C \ | \ X \otimes_C T = 0}$ and then there exists an indecomposable $A$-module $N$ such that $\Hom_A(T,N)=0$ and $\Ext^1_A(T,N) \simeq Z$ so that the above almost split sequence is a connecting sequence and $T_0$ is injective, a contradiction. This completes the proof that $Z \in \mathcal Y(T)$ and hence there exists $M \in \T(T)$ such that $Z= \Hom_A(T,M)$. The almost split sequence
		$$0 \fl \Hom_A(T,T_0) \fl \Hom_A(T,T^*) \fl \Hom_A(T,M) \fl 0$$
		yields, upon applying $- \otimes_C T$, a short exact sequence in $\modd A$
		$$0 \fl T_0 \fl T^* \fl M \fl 0.$$
		This completes the proof that the morphism $T_0 \fl T^*$ is injective and therefore the almost complete module $\overline T$ is sincere.

		As recalled in Section \ref{section:background}, there exists exactly one indecomposable module $T_0' \not \simeq T_0$ such that $T' = \overline T \oplus T_0'$ is a tilting module and moreover by \cite{HU:almostcomplete}, there exists a non-split short exact sequence
		\begin{equation}\label{eq:ses}
			0 \fl T_0 \fl E \fl T_0' \fl 0
		\end{equation}
		with $E \in \add \overline T$. The existence of such a sequence shows that $T_0'$ is a successor of a preinjective module, and hence is preinjective. Therefore $T'$ is preinjective as well.
	
		Now we claim that $\T(T') \subsetneqq \T(T)$. Let $M \in \T(T')$, then $\Ext^1_A(T',M)=0$. Applying the functor $\Hom_A(-,M)$ to the sequence \eqref{eq:ses}, we get an epimorphism
		$$0 = \Ext^1_A(E,M) \fl \Ext^1_A(T_0,M) \fl 0$$
		because $A$ is hereditary. Hence $\Ext^1_A(T_0,M) = 0$. Since $\Ext^1_A(\overline T, M)=0$ we have $\Ext^1_A(T,M) = 0$ and so $M \in \T(T)$. This shows that $\T(T') \subset \T(T)$. On the other hand, $T_0 \not \in \T(T')$ because otherwise, there exists a direct summand $T_1$ of $T'$ such that there exists a non-zero map $T_1 \fl T_0$. Now such a summand cannot be $T_0'$ because otherwise we would have a cycle $T_0' \fl T_0 \fl * \fl T_0'$ with $T_0'$ and $T_0$ preinjective, an absurdity. Hence $T_1 \in \add \overline T$ and so $T_1 \in \add T$ which contradicts the choice of $T_0$. Therefore we indeed have $T_0 \not \in \T(T')$. Since $T_0 \in \T(T)$, this completes the proof that $\T(T') \subsetneqq \T(T)$. Setting $T = T^{(0)}$, $T' = T^{(1)}$, the proof of the proposition is completed by induction on $|\ind-\T(T^{(i)})|$.
	\end{proof}

	\subsection{Proof of Proposition \ref{prop:concealed}}
		The proof is similar to the proof of Theorem \ref{theorem:transjective}. Let $\CC$ be a cluster category and $T$ a tilting object in $\CC$ with $T$ transjective. Let $\overline T = \bigoplus_{i \neq k} T_i$. Assume that the mutation $T' = \overline T \oplus T_k^*$ is also transjective. Let $M$ and $N$ be two indecomposable rigid objects in $\CC$. If $M$ is transjective, then it follows from \cite[Corollary 4.2]{BMR3} that it is compatible with any exchange pair in $\CC$. If $M$ is regular, then the proof of \cite[Lemma 3.5]{BM:affine} applies and it follows that $M$ is compatible with any exchange pair $(X,X^*)$ where $X$ and $X^*$ are transjective. In particular, $M$ is compatible with the exchange pair $(\Sigma^2 T_k,\Sigma^2 T_k^*)$.

		Since $\Hom_{\CC}(T,-)$ induces an equivalence $\CC/(\add(\Sigma T)) \xrightarrow{\sim} \modd B$, it is enough to prove that if $M, N \not \in \add(\Sigma T)$ are such that $\ddim \Hom_{\CC}(T,M) = \ddim \Hom_{\CC}(T,N)$, then $M \simeq N$ in $\CC$. According to Proposition \ref{prop:muttransjective}, there exists a sequence of mutations $T \fl T' \fl \cdots \fl T^{(l)} =\kQ$ joining $T$ to $\kQ$ in $\CC$ such that every tilting object arising in the sequence is transjective. The rest of the proof follows by induction using Lemma \ref{lem:induction}, as in the proof of Theorem \ref{theorem:transjective}. \hfill \qed

	\subsection{Proof of Corollary \ref{corol:affineconcealed}}
		In order to prove Corollary \ref{corol:affineconcealed}, it is enough to prove that for a cluster-concealed algebra which is either representation-finite or tame, any indecomposable rigid $B$-module $M$ lifts to an indecomposable rigid object $\widetilde M$ in the corresponding cluster category. Let thus $Q$ be a Dynkin or an affine quiver and let $\CC$ denote the cluster category of $Q$. Assume that $B = \End_{\CC}(T)$ for some tilting object $T$ in $\CC$. 

		If $M$ is a transjective $B$-module, its lift is a transjective object in $\CC$ which is therefore rigid, proving the claim. If $Q$ is a Dynkin quiver, then any object in $\CC$ is transjective so we are done. Assume therefore that $Q$ is an affine quiver and that $M$ is not transjective. Then its lift $\widetilde M$ lies in an exceptional tube $\mathcal T$ of the Auslander-Reiten quiver $\Gamma(\CC)$ of $\CC$. The tube containing $M$ in the Auslander-Reiten quiver of $\modd B$ is standard. Similarly, the tube containing $F_{\kQ}\widetilde M  = \Hom_{\CC}(kQ,\widetilde M)$ in the Auslander-Reiten quiver of $\modd \kQ$ is standard. Moreover, these two tubes are identified with $\mathcal T$, while both objects $M$ and $F_{\kQ}\widetilde M$ correspond to the point $\widetilde M$ in that tube. Since $M$ is rigid, it follows that $F_{\kQ}\widetilde M$ is rigid in $\modd \kQ$. Therefore, it follows from \cite{BMRRT} that $\Ext^1_{\CC}(\widetilde M,\widetilde M) \simeq \Ext^1_{\kQ}(F_{\kQ}\widetilde M,F_{\kQ}\widetilde M) \oplus D\Ext^1_{\kQ}(F_{\kQ}\widetilde M,F_{\kQ}\widetilde M) = 0$, proving the claim. The corollary then follows from Proposition \ref{prop:concealed}.

	\subsection{Proof of Corollary \ref{corol:denaffine}}
		The proof is similar to the proof of Corollary \ref{corol:dentransj}. We assume that $\mathcal A$ is a coefficient-free cluster algebra associated with a quiver mutation-equivalent to an affine quiver $Q$. We consider two cluster variables $x$ and $y$ in $\mathcal A$ which have the same denominator vector in a cluster $\mathbf c$ corresponding to a transjective tilting object $T$ and we let $M$ and $N$ be indecomposable rigid objects in $\CC$ such that $x = X^T_M$ and $y = X^T_N$. By assumption, we have $\den(X^T_M) = \den(X^T_N)$. 

		If $M$ or $N$ is in $\add(\Sigma T)$, then $X^T_M$ and $X^T_N$ are cluster variables in $\mathbf c$ and therefore, $X^T_M = X^T_N$. If they are not in $\add(\Sigma T)$, since $T$ is transjective, then it follows from \cite{BMR3} that $\den(X^T_M)= \ddim \Hom_{\CC}(T,M)$ and $\den(X^T_N)= \ddim \Hom_{\CC}(T,N)$. Therefore, $\Hom_{\CC}(T,M)$ and $\Hom_{\CC}(T,N)$ are indecomposable rigid transjective $B$-modules lifting to indecomposable rigid objects in $\CC$ and they have the same dimension vector. Thus, Proposition \ref{prop:concealed} implies that $\Hom_{\CC}(T,M) = \Hom_{\CC}(T,N)$ and therefore $M \simeq N$ in $\CC$, so that $x = X^T_M = X^T_N =y$. \hfill \qed

\section*{Acknowledgements}
	The first author gratefully acknowledges support from the NSERC of Canada, the FQRNT of Qu\'ebec and the Universit\'e de Sherbrooke.

	The second author gratefully acknowledges support from the ANR \emph{G\'eom\'etrie Tropicale et Alg\`ebres Amass\'ees}. He would also like to thank the first author for his kind hospitality during his stay at the Universit\'e de Sherbrooke where this research was done.

\newcommand{\etalchar}[1]{$^{#1}$}


\end{document}